\documentclass[runningheads]{llncs}
\usepackage[T1]{fontenc}
\usepackage{graphicx}
\usepackage{amsmath,amssymb}
\usepackage{mathtools}
\usepackage{booktabs}
\usepackage{multirow}
\usepackage{xcolor}
\usepackage[normalem]{ulem}
\usepackage{siunitx}
\usepackage{enumitem}
\usepackage{tikz}
\usepackage{subcaption}

\def\div{\operatorname{div}}
\def\NN{\mathbb{N}}
\def\RR{\mathbb{R}}
\def\A{\mathcal{A}}
\def\K{\mathcal{K}}
\def\M{\mathcal{M}}
\def\R{\mathcal{R}}
\def\V{\mathcal{V}}
\def\H{\mathcal{H}}
\def\dt{\partial_t}
\def\dtau{d_\tau}
\newcommand{\fmtsec}[1]{\SI[round-precision=2,round-mode=places]{#1}{\second}}

\begin{document}
\title{A parallel-in-time solver for nonlinear\\degenerate time-periodic parabolic problems}
\titlerunning{Parallel-in-time solver for nonlinear time-periodic problems}
\author{Herbert Egger\inst{1,2}
\and
Andreas Schafelner\inst{1}
}
\authorrunning{H. Egger and A. Schafelner}
\institute{$^1$Institute of Numerical Mathematics, Johannes Kepler University Linz\\
$^2$Johann Radon Institute for Computational and Applied Mathematics\\
\email{herbert.egger@jku.at,andreas.schafelner@jku.at}}
\maketitle          
\begin{abstract}
A class of abstract nonlinear time-periodic evolution problems is considered which arise in electrical engineering and other scientific disciplines. 
An efficient solver is proposed for the systems arising after discretization in time based on a fixed-point iteration. Every step of this iteration amounts to the solution of a discretized time-periodic and time-invariant problem for which efficient parallel-in-time methods are available. Global convergence with contraction factors independent of the discretization parameters is established. Together with an appropriate initialization step, a highly efficient and reliable solver is obtained. The applicability and performance of the proposed method is illustrated by simulations of a power transformer. Further comparison is made with other solution strategies proposed in the literature.
\keywords{Parallel-in-time methods \and Nonlinear time-periodic problems \and Mesh-independent convergence.}
\end{abstract}

\section{Introduction}
\label{sec:intro}

We consider the efficient numerical solution of degenerate nonlinear time-periodic problems of the general abstract form
\begin{alignat}{2}
\M \partial_t u(t) + \K(u(t)) &= f(t) \qquad && \text{in } \V^*, \ t \in [0,T] \label{eq:sys1}\\
\M u(0) &= \M u(T) \qquad && \text{in } \H^*. \label{eq:sys2}
\end{alignat}
Here $\V \subset \H$ are two real Hilbert spaces, $\H^* \subset \V^*$ the corresponding dual spaces, $\K : \V \to \V^*$ is assumed strongly monotone and Lipschitz, and $\M : \H \to \H^*$ to be continuous, symmetric, and non-negative. Existence of a unique solution to such problems can be established by monotonicity arguments; see e.g.~\cite{Showalter:1997,Zeidler:1990a}.

Systems of this kind have important applications of electrical engineering, e.g., in the modelling of electric motors and power transformers, and much research has been devoted to their analysis and efficient numerical solution. 
Results about the existence of a unique solution to the time-harmonic magneto-quasistatic problem can be found in \cite{BachingerLangerSchoeberl:2005a,Ceserano:2024} based on non-constructive arguments. 
The first reference further studies the numerical approximation by a Fourier-Galerkin method in time together with a finite element discretization in space. Related results can be found under the name \emph{Harmonic Balance} finite element methods in the engineering literature~\cite{Yamada:1988,Albanese:1992,Gyselink:2002}. 
While this approach seems well-suited from an approximation point of view, its implementation requires the solution of a large fully-coupled and highly nonlinear finite element system which seems to limit the actual applicability in practice~\cite{Scolaro:2024}. 
Discretizing \eqref{eq:sys1}--\eqref{eq:sys2} in time, e.g., by the implicit Euler method, leads to a loosely coupled system of nonlinear elliptic problems, whose solution can be tackled by fixed-point or Newton methods \cite{Biro:2006,Dlala:2007,Takahashi:2012}. A comparison of the approaches can be found in \cite{Plasser:2018}.
The efficient solution of \emph{linear} time-invariant methods by the \emph{Parareal} and \emph{Multigrid-in-Time} methods has be demonstrated in \cite{Gander:2013,Kulchytska:2021a} and \cite{GanderNeumueller:2016}. 
To the best of our knowledge, a rigorous convergence analysis of iterative solvers for infinite-dimensional degenerate nonlinear time-periodic problems is not available to date.

In this paper, we propose an efficient iterative solution strategy for time-periodic problems of the form \eqref{eq:sys1}--\eqref{eq:sys2}.  Similar to \cite{Biro:2006}, the system is first discretized in time and then solved iteratively by a fixed-point iteration. Every step of this iteration requires the solution of a time-periodic and time-invariant linear system. 
The well-posedness of this iteration as well as its global convergence, independent of discretization parameters, is established. 
The individual steps of the fixed-point iteration can finally be realized efficiently by the multigrid-in-time strategy proposed in \cite{GanderNeumueller:2016}. 
Together with a careful initialization strategy, we obtain a provably convergent and highly efficient parallel-in-time solver.
The main arguments used in our proofs also apply to the time-continuous problem. As a by-product of our analysis, we thus also obtain a constructive proof for the existence and uniqueness of solutions to the time-continuous system \eqref{eq:sys1}--\eqref{eq:sys2}.

\section{Preliminaries}
\label{sec:prelim}

Let $\V \subset \H$ be real Hilbert spaces with continuous embedding. We identify $\H$ with its dual $\H^*$ and obtain a Gelfand triple $\V \subset \H \subset \V^*$. 
We assume that 
\begin{itemize}[leftmargin=2.5em,topsep=0.5em,parsep=0.5em]
\item[(A1)] $\K : \V \to \V^*$ is strongly monotone and Lipschitz, i.e., 
$\langle \K(u)-\K(v),u-v\rangle \ge \gamma \|u-v\|^2_V$  and $\|\K(u) - \K(v)\|_{V^*} \le L \|u-v\|_V$
for all $u,v \in \V$ with $L,\gamma>0$; 
\item[(A2)] $\M : \H \to \H^*$ is continuous, symmetric, and non-negative, i.e., $\|\M u\|_H \le C \|u\|_\H$, $\langle \M u,v\rangle = \langle u,\M v\rangle$, and $\langle \M u,u\rangle \ge 0$ for all $u,v \in \H$. 
\end{itemize}
These general assumptions already allow to establish the following result. 
\begin{theorem} \label{thm:sys}
For any $T>0$ and $f \in L^2(0,T;V^*)$, the system \eqref{eq:sys1}--\eqref{eq:sys2} has a unique solution $u \in L^2(0,T;V)$ with $\M \dt u \in L^2(0,T;V^*)$.
\end{theorem}
A proof of this assertion can be deduced from \cite{Showalter:1997,Zeidler:1990a}; also see Remark~\ref{rem:fixedpoint} below. 

Now let $T>0$ and $N \in \NN$ be given. We define $\tau=T/N$, set $t^n = n \tau$, and denote by $\dtau u^n = \frac{1}{\tau} (u^n - u^{n-1})$ the backward difference quotient in time. In the rest of the paper, we consider the following semi-discrete problem. 
\begin{problem} \label{prob:disc}
Given $(f^1,\ldots,f^N) \subset \V^*$, find $(u^1,\ldots,u^N) \subset \V$ such that 
\begin{align}
\M \dtau u^n + \K(u^n) &= f^n, \qquad n=1,\ldots,N, \label{eq:disc1} \\
\text{with} \qquad \qquad 
u^0 &= u^N. \label{eq:disc2}
\end{align}
The second equation is used to implicitly define $\dtau u^1 = \frac{1}{\tau}(u^1 - u^0) = \frac{1}{\tau}(u^1 - u^N)$.
\end{problem}

\begin{theorem} \label{thm:disc}
Let (A1)--(A2) hold. Then Problem~\ref{prob:disc} has a unique solution.
\end{theorem}
\begin{proof}
We introduce $U=(u^1,\ldots,u^N)$ and $F=(f^1,\ldots,f^n)$ for abbreviation. Then Problem~\ref{prob:disc} can be written equivalently as
\begin{align} \label{eq:disc}
\text{Find } U \in \V^N : \qquad \A(U)=F \qquad \text{in } (\V^*)^N 
\end{align}
with nonlinear operator $\A : \V^N \to (\V^N)^*$ defined by $\A(U)^n = \M \dtau u^n + \K(u^n)$. 
It is known that $\dtau |a^n|^2 = 2 (\dtau a^n,a^n) - \tau |\dtau a^n|^2 \le 2 (\dtau a^n,a^n)$. Together with the properties of $\M$ and $\K$, we thus obtain 
\begin{align*}
\langle \A(U) &- \A(V), U-V \rangle \\
&=\sum\nolimits_{n=1}^N \langle \M \dt (u^n - v^n),u^n - v^n\rangle + \langle \K(u^n) - \K(v^n), u^n-v^n\rangle \\
&\ge \sum\nolimits_{n=1}^N \tfrac{1}{2} \dtau (\M (u^n-v^n),u^n-v^n) + \gamma \|u^n-v^n\|_\V^2 
= \|U - \V\|_{\V^N}^2.
\end{align*}
We used $\|U\|_{\V^N}^2 = \sum_n \|u^n\|_\V^2$ to abbreviate the squared norm on $\V^N$. This shows that $\A$ is strongly monotone. In a similar manner, one can see that 
\begin{align*}
\|\A(U)-\A(V)\|_{(\V^*)^N} 
\le \left(\tfrac{2C}{\tau} + L \right) \|U - V\|_{\V^N}.
\end{align*}
Hence $\A : \V^N \to (\V^*)^N$ is a strongly monotone and Lipschitz continuous operator. Existence of a unique solution to $\A(U) = F$ then follows from Zarantonello's fixed-point theorem; see \cite[Thm.~25.B]{Zeidler:1990a}. 
\qed
\end{proof}
The proof of Zarantonello's fixed-point theorem, and hence also of the previous result, provides an iterative method for the solution of \eqref{eq:disc1}--\eqref{eq:disc2}. 
The contraction factor however dependence on the ratio of the Lipschitz- and monotonicity constants, and hence degenerates with the time-step $\tau$ going to zero. 
In the following section, we propose an alternative iterative scheme which converges uniformly.

\section{Iterative solution strategy}
\label{sec:iter}

Let $\widehat \K : \V \to \V^*$ be a linear and continuous operator, and $(u^{1,0},\ldots,u^{N,0}) \in \V^N$ some given initial iterate.
For the numerical solution of \eqref{eq:disc1}--\eqref{eq:disc2}, we consider the fixed-point iteration defined by
\begin{alignat}{2}
\M \dtau u^{n,k+1} + \widehat \K u^{n,k+1} &= f^n + \widehat \K u^{n,k} - \K(u^{n,k}), \quad && n=1,\ldots,N, \label{eq:iter1} \\
\text{with} \qquad \qquad 
u^{0,k+1} &= u^{N,k+1}, && k \ge 0. \label{eq:iter2}
\end{alignat}
Abbreviating $U^k=(u^{1,k},\ldots,u^{N,k})$ and $U^{k+1}=(u^{1,k+1},\ldots,u^{N,k+1})$ as before, the system \eqref{eq:iter1}--\eqref{eq:iter2} can be written as fixed-point iteration $U^{k+1} = \Phi(u^k)$ with mapping $\Phi : \V^N \to \V^N$, $U^k \mapsto \Phi(U^k)=U^{k+1}$ defined implicitly via the equations above.
Convergence of this iteration is guaranteed by the following result. 
\begin{theorem} \label{thm:iter}
Let (A1)--(A2) hold and assume that 
$\widehat \K : \V \to \V^*$ is a linear, continuous, symmetric, and elliptic operator which satisfies 
\begin{align} \label{eq:A3} %\tag{A3}
\|\widehat \K (u-v) - (\K(u) - \K(v))\|_{\widehat \K^{-1}} \le q \|u-v\|_{\widehat \K} \qquad \forall u,v \in \V
\end{align}
with some $q < 1$. Then for any $U^0=(u^{1,0},\ldots,u^{N,0}) \in \V^N$, the iteration \eqref{eq:iter1}--\eqref{eq:iter2} converges linearly to the unique solution $U=(u^1,\ldots,u^N)$ of \eqref{eq:disc1}--\eqref{eq:disc2}, i.e., 
\begin{align} \label{eq:q}
\|U^{k+1}-U\|_{\widehat \K} \le q \, \|U^k-U\|_{\widehat \K}, \qquad\forall k \ge 0.
\end{align}
Here $\|U\|_{\widehat \K}^2 = \sum_{n=1}^N\|u^n\|_{\widehat \K}^2$ and $\|u\|_{\widehat \K}^2 = \langle \widehat \K u,u\rangle$ denote the square of the norms on $\V^N$ and $\V$ induced by the operator $\widehat K$, and similarly $\|F\|_{\widehat \K^{-1}}^2 = \sum_{n=1}^N\|f^n\|_{\widehat \K^{-1}}^2$ and $\|f\|_{\widehat \K^{-1}}^2 = \langle \widehat \K^{-1} f,f\rangle$ define corresponding squared norms on $(\V^*)^N$ and $\V^*$. 
\end{theorem}
\begin{proof}
The operator $\widehat \K : \V \to \V^*$ is continuous and elliptic, and hence satisfies the assumption (A1)--(A2). We may thus apply Theorem~\ref{thm:disc} to see that the iteration \eqref{eq:iter1}--\eqref{eq:iter2} is well-defined. In particular, $\Phi : \V^N \to \V^N$, $U \to \Phi(U)$ is a self-mapping on $\V^N$.
Similar to the proof of Theorem~\ref{thm:disc}, we see that 
\begin{align*}
\|U^{k+1} - U\|_{\K}^2 
&\le \sum\nolimits_{n=1}^N \langle \dtau (u^n -v^n), u^n-v^n \rangle + \langle \widehat \K (u^{n,k+1} - u^n), u^{n,k+1} - u^n \rangle \\
&= \sum\nolimits_{n=1}^N \langle \widehat \K (u^{n,k} - u^n) - (\K(u^{n,k}) - \K(u^n)), u^{n,k+1} - u^n\rangle. % = (*).
\end{align*}
In the second step, we used the definitions of $U^{k+1}$ and $U$ provided by \eqref{eq:iter1}--\eqref{eq:iter2} and \eqref{eq:disc1}--\eqref{eq:disc2}, respectively.  
By application of \eqref{eq:A3}, we  obtain 
\begin{align*}
\langle \widehat \K (u^{n,k} &- u^n) - (\K(u^{n,k}) - \K(u^n)), u^{n,k+1} - u^n\rangle \\
&\le \|\widehat \K (u^{n,k} - u^n) - (\K(u^{n,k}) - \K(u^n))\|_{\widehat \K^{-1}} \|u^{n,k+1} - u^n\|_{\widehat \K} \\
&\le q \|u^{n,k} - u^n\|_{\widehat \K} \|u^{n,k+1}-u^n\|_{\widehat \K}. 
\end{align*}
Inserting this into the previous formula and applying a Cauchy-Schwarz inequality for the sum, we arrive at 
\begin{align*}
\|U^{k+1} - U\|_{\widehat \K}^2 \le q \|U^{k}-U\|_{\widehat \K } \|U^{k+1} - U\|_{\widehat \K}
\end{align*}
which implies the estimate \eqref{eq:q} and the remaining assertions of the theorem. \qed
\end{proof}

\begin{remark} \label{rem:contraction}
Let $\R : \V^* \to \V$ be the Riesz-isomorphism on the Hilbert space $\V$ defined by $(\R f,v)_{\V} = \langle f,v\rangle$ for all $v \in \V$. 
Further set $\widehat \K = \frac{1}{\omega} \R^{-1}$ with $\omega>0$ to be determined. Then we can express
$\|u\|_{\widehat \K}^2 = \omega \|u\|_\V^2$ and obtain
\begin{align*}
\|\widehat \K (u-v) &- (\K(u) - \K(v))\|_{\widehat \K^{-1}}^2 \\
&= \|u-v\|_{\widehat \K}^2 - 2 \langle \K(u) - \K(v), u-v \rangle + \|\K(u) - \K(v)\|^2_{\widehat \K^{-1}} \\
&\le \|u-v\|_{\widehat \K}^2 (1 - 2 \omega \gamma + \omega^2 L^2).
\end{align*}
In the second step we used the monotonicity and Lipschitz estimates provided by assumption (A1) and the relation between the norms stated above. For any $ 0 < \omega < L^2/\gamma$, the factor $q(\omega)^2=1-2 \omega \gamma + \omega L^2$ is smaller than one. This shows that condition \eqref{eq:q} can always be satisfied with a uniform constant $q<1$. 
\end{remark}

\begin{remark} \label{rem:parallel}
Every step of the fixed-point method \eqref{eq:iter1}--\eqref{eq:iter2} requires the solution of a semi-discrete time-periodic and time-invariant system. Via discrete Fourier transforms, this system can be converted into block-diagonal form with complex symmetric operators $i \omega^k \M + \widehat \K$ on the diagonal. Such a strategy was used in~\cite{Biro:2006}. 
Alternatively, the systems can be solved by multigrid-in-time algorithms~\cite{GanderNeumueller:2016,Vandewalle92}. 
Both strategies essentially lead to parallel-in-time algorithms. 
\end{remark}

\begin{remark} \label{rem:fixedpoint}
The arguments used in the previous proof can in principle also be applied  to the original system \eqref{eq:sys1}--\eqref{eq:sys2}. This leads to the fixed-point iteration 
\begin{alignat*}{2}
\M \dt u^{k+1}(t) + \widehat \K(u^{k+1}(t)) &= f(t) + \widehat \K u^{k}(t) - \K(u^k(t)), \qquad && 0 \le t \le T \\
\M u^{k+1}(0) &= \M u^{k+1}(0). 
\end{alignat*}
Every step of this iteration amounts to the solution of a linear time-periodic and time-invariant problem. Well-posedness of the iteration can then be obtained from standard result \cite{Showalter:1997}, and convergence of the iterates in the space $L^2(0,T;V)$ follows with the same reasoning as above. 
\end{remark}

\section{Numerical illustration}
\label{sec:num}

In order to illustrate the applicability and efficiency of the proposed method, we consider the time-periodic magnetic fields in a power transformer. The problem setup is adopted from \cite{Plasser:2018}, but to simplify the reproducibility of the results, we here consider only a two-dimensional setting. 

\medskip 
\noindent 
\textbf{Test problem.}
The domain $\Omega \subset \RR^2$ represents the cross-section of the transformer, which is covered by different materials; see  Figure~\ref{fig:domain} for a sketch. 
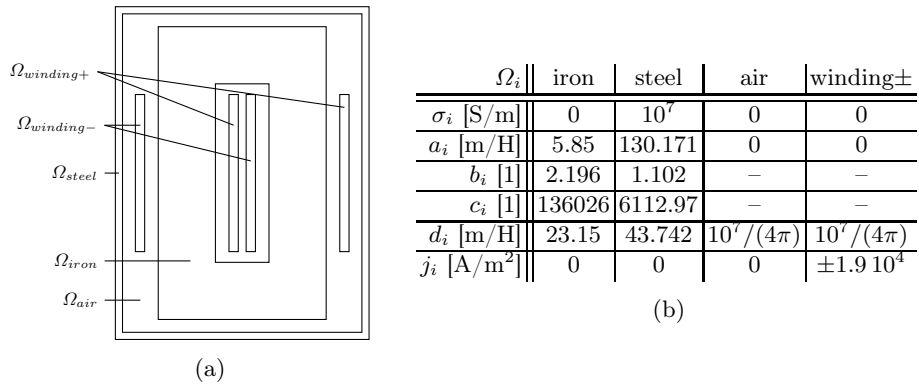
\begin{figure}[ht!]
    \centering%
    \begin{subfigure}[c]{.45\textwidth}
        \begin{tikzpicture}[scale=.095]%
            % outer 
            \draw (0,0) -- (355mm, 0) -- (355mm,466mm) -- (0,466mm) -- cycle;
            % steel inner
            \draw (10mm, 10mm) -- (345mm, 10mm) -- (345mm,456mm) -- (10mm, 456mm) -- cycle;
            % iron core
            \draw (60mm, 28mm) -- (295mm, 28mm) -- (295mm, 438mm) -- (60mm, 438mm) -- cycle;
            \draw (140mm, 108mm) --  (215mm, 108mm) -- (215mm, 358mm) -- (140mm, 358mm) -- cycle;
            % outer wiring
            \draw (28mm, 123mm) --  ++(13mm, 0) --  ++(0, 220mm) --  ++(-13mm, 0) -- cycle;
            \draw (159mm, 123mm) --  ++(13mm, 0) --  ++(0, 220mm) --  ++(-13mm, 0) -- cycle;
            \draw (183mm, 123mm) --  ++(13mm, 0) --  ++(0, 220mm) --  ++(-13mm, 0) -- cycle;
            \draw (314mm, 123mm) --  ++(13mm, 0) --  ++(0, 220mm) --  ++(-13mm, 0) -- cycle;
            % labels
            \draw (5mm, 233mm) -- ++(-2cm, 0) node[left,scale=.75] {$ \Omega_{steel} $};
            \draw (105mm, 110mm) -- ++(-12cm, 0) node[left,scale=.75] {$ \Omega_{iron} $};
            \draw (35mm, 55mm) -- ++(-5cm, 0) node[left,scale=.75] {$ \Omega_{air} $};
            \draw (165.5mm, 300mm) -- (-25mm, 375mm) node[left,scale=.75] {$ \Omega_{winding+} $};
            \draw (320.5mm, 325mm) -- (-25mm, 375mm);
            \draw (35mm, 300mm) -- ++(-5cm, 0) node[left,scale=.75] {$ \Omega_{winding-} $};
            \draw (189.5mm, 250mm) -- ++(-20.5cm, 5cm);
        \end{tikzpicture}%
        \caption{}
        \label{fig:domain}%
    \end{subfigure}%
    \begin{subfigure}[c]{.55\textwidth}
        \begin{tabular}{r||c|c|c|c}%
            $ \Omega_i $        & iron & steel & air & winding$\pm$ \\ \hline \hline
            $\sigma_i$ [S/m]    & 0 & $10^7$ & 0 & 0  \\ \hline
            $a_i$ [m/H]         & $5.85$ & $130.171$& 0 & 0  \\ \hline
            $b_i$ [1]            & $2.196$ & $1.102$ & -- & -- \\ \hline
            $c_i$ [1]            & $136026$ & $6112.97$ & -- & -- \\ \hline
            $d_i$ [m/H]         & $23.15$  &  $43.742$ & $ 10^7 / (4\pi) $ & $ 10^7 / (4\pi) $\\ \hline
            $j_i$ [A/m$^2$]     & 0  & 0  &      0             & $\pm 1.9\,10^{4} $ 
        \end{tabular}%
        \caption{}
        \label{tab:coef}
    \end{subfigure}%
    \caption{Cross-section of transformer (left); 
    table with material parameters (right).}%
    \vspace*{-30pt}
\end{figure}
The $z$-component $u$ of the magnetic vector potential is governed by 
\begin{alignat*}{2}
    \sigma \partial_t u - \div ( \nu(|\nabla u|) \nabla u(t) ) &= j \qquad &&\text{in } \Omega, \ t \in [0,T]\\
    u &= 0 \qquad &&\text{on } \partial\Omega, \ t \in [0,T]
\end{alignat*}
with $j$ denoting the $z$-component of the current in the windings. All fields are assumed time-periodic with period $T= 0.02 $ [s]. 
The electric conductivity $\sigma$ and magnetic reluctivity~$\nu$ are assumed to be of the form 
\begin{align*}
\sigma = \sigma_i 
\qquad \text{and} \qquad 
\nu(s) = a_i \min\{\exp( b_i s^2 ),c_i\} + d_i \qquad \text{in } \Omega_i
\end{align*}
with material dependent constants $\sigma_i,\ldots,d_i$. The current density is given by 
\begin{align*}
    j(t)= j_i \, \cos(2\pi t / T) 
\end{align*}
with piecewise constant current amplitudes $j_i$. The values for the coefficients used in our simulations are summarized in Figure~\ref{tab:coef}.

\medskip 

\noindent 
\textbf{Simulation setup.}
The problem is discretized in space by piecewise linear finite elements over a triangular mesh. The dimension of the corresponding finite element spaces is denoted by $N_V$. 
For time-discretization, we use the implicit Euler method with constant time step $\tau = T/N$. 
This leads to a finite-dimensional version of \eqref{eq:disc1}--\eqref{eq:disc2}. 
According to our analysis, we expect convergence of the proposed iterative solvers with iterations numbers independent of $N_V$ and $N$. 
In our numerical tests, we compare the following three solution strategies: 
\begin{itemize}[leftmargin=2.5em,topsep=0.5em,parsep=0.5em]
\item[(M1)] the fixed-point iteration proposed in this paper; the linearized systems \eqref{eq:iter1}--\eqref{eq:iter2} are solved by the multigrid-in-time method proposed in \cite{GanderNeumueller:2016};
\item[(M2)] the Newton-method applied to \eqref{eq:disc1}--\eqref{eq:disc2}; the resulting linearized systems are solved by the multigrid-in-time method proposed in \cite{GanderNeumueller:2016}; 
Armijo back-tracking is applied to guarantee decay in the residual;
\item[(M3)] a simple nonlinear time-stepping method; the iteration \eqref{eq:disc1} is performed for $n \ge 0$ until the periodic limit cycle is reached; the nonlinear systems in every time-step are solved by a Newton method with Armijo back-tracking.
\end{itemize}
\textbf{Initialization.}
All iterations are initialized by solving the decoupled nonlinear static problems that result from setting $\sigma \equiv 0$. 
The iterations are stopped when the residual has been reduced sufficiently, i.e., when 
\begin{align*} 
\|F - \A(U^k)\|_{\ell^2} \le tol \, \|F - \A(U^0)\|_{\ell^2}.
\end{align*}
Here $F-\A(U)$ amounts to the residual vector for the finite element approximation of \eqref{eq:disc1}--\eqref{eq:disc2}. 
For our computational tests, we choose the tolerance $tol=10^{-4}$. 

\medskip 

\noindent 
\textbf{Numerical results.}
The methods (M1)--(M3) are implemented using the finite element library MFEM 
\cite{mfem}, with the spatial solvers provided by \emph{Hypre}\footnote{\url{https://github.com/hypre-space/hypre}} and \emph{SuiteSparse}\footnote{\url{https://github.com/DrTimothyAldenDavis/SuiteSparse}}. 
The comparison runs are performed on a single node of Dane\footnote{\url{https://hpc.llnl.gov/hardware/compute-platforms/dane}}, which consists of a 112-core Intel Sapphire Rapid processor, while
the scaling tests were performed on multiple nodes of the same platform.
We present the number of iterations for the methods (M1)--(M3) in Table~\ref{tab:runtime}, where the corresponding wall time is included in parentheses. Iterations marked with $^*$ were stopped early due to slow convergence. 
We also report the run times for the initialization step, which  requires the solution of a sequence of decoupled nonlinear elliptic problems. Methods (M1) and (M2) use 8 MPI ranks, while method (M3) utilizes multi-threading for the linear solvers. In Table~\ref{tab:scaling}, we present the results of a parallel scaling study for (M1), for a problem
with $N_V = 8208$ spatial dofs. 
\begin{table}[ht!]
    \vspace*{-10pt}
    \begin{subtable}{\textwidth}
        \centering%
        \setlength\tabcolsep{3.6mm}
         \begin{tabular}{rlllllllll}\toprule
            $N$ & $N_V$ & init  & (M1) & (M2) & (M3) \\ \midrule
            128 & 8208  & \fmtsec{28.8542}  & \textbf{13} (\fmtsec{8.25036}) & \textbf{5} (\fmtsec{12.7918}) & \textbf{10}$^*$ (\fmtsec{153.181}) \\
            128 & 32451 & \fmtsec{190.374}  & \textbf{10} (\fmtsec{24.6771}) & \textbf{5} (\fmtsec{49.1059}) & \textbf{10}$^*$ (\fmtsec{647.643}) \\  \midrule
            256 & 8208  & \fmtsec{58.3237}  & \textbf{13} (\fmtsec{16.0703}) & \textbf{5} (\fmtsec{24.4229}) & \textbf{10}$^*$ (\fmtsec{306.23})  \\
            256 & 32451 & \fmtsec{406.597}  & \textbf{10} (\fmtsec{49.6497}) & \textbf{5} (\fmtsec{94.2119}) & \textbf{10}$^*$ (\fmtsec{1288.96}) \\ \midrule
            512 & 8208  & \fmtsec{135.188}  & \textbf{13} (\fmtsec{30.3863}) & \textbf{5} (\fmtsec{48.0321}) & \textbf{10}$^*$ (\fmtsec{617.965}) \\    
            512 & 32451 & \fmtsec{828.546}  & \textbf{10} (\fmtsec{98.5296}) & \textbf{5} (\fmtsec{187.653}) & \textbf{10}$^*$ (\fmtsec{2577.77}) \\
            \bottomrule
        \end{tabular}
        \caption{Comparison of methods (M1) -- (M3) for different discretization parameters.}
        \label{tab:runtime}
    \end{subtable}
    \begin{subtable}{\textwidth}
        \centering%
    \setlength\tabcolsep{1.4mm}
    \begin{tabular}{rllllllll}\toprule
        $N \, \backslash \, R$ & 16 & 32 & 64 & 128 & 256 & 512  \\ 
                             \midrule
        256                  & \textbf{13} (\fmtsec{9.20935}) & \textbf{13} (\fmtsec{5.29112}) & \textbf{13} (\fmtsec{3.45101}) & \textbf{13} (\fmtsec{2.66415}) & --                             & -- \\
        512                  & \textbf{13} (\fmtsec{16.8417}) & \textbf{13} (\fmtsec{9.18293}) & \textbf{13} (\fmtsec{5.37856}) & \textbf{13} (\fmtsec{3.59071}) & \textbf{13} (\fmtsec{3.03495}) & -- \\
        1024                 & \textbf{13} (\fmtsec{31.0817}) & \textbf{13} (\fmtsec{16.3910}) & \textbf{13} (\fmtsec{9.02057}) & \textbf{13} (\fmtsec{5.43837}) & \textbf{13} (\fmtsec{4.03325}) & \textbf{13} (\fmtsec{3.78642})\\
        2048                 & --                             & --                             & \textbf{13} (\fmtsec{16.423})  & \textbf{13} (\fmtsec{9.13183}) & \textbf{13} (\fmtsec{6.05007}) & \textbf{13} (\fmtsec{5.04528})\\
        \bottomrule
    \end{tabular}
    \caption{Weak and strong scaling results for (M1) with $N_V = 8208$ and MPI ranks $R$.}
    \label{tab:scaling}
    \end{subtable}
    \caption{Iteration counts and wall times.}
    \label{tab:results}
    \vspace*{-40pt}
\end{table}

\bigskip 

\noindent 
\textbf{Discussion of the numerical results.}
As predicted by our theoretical results, the proposed method (M1) converges with an almost fixed number of iteration, independent of the discretization parameters. A similar observation holds for the Newton method (M2). 
Although the latter requires fewer outer iterations, 
the solution of the linearized time-periodic problems in every step of the iteration is somewhat faster for method (M1), and hence the overall computation times are smaller. 
Let us note that convergence of the multigrid-in-time method for the Newton system is not fully covered by theory yet. An analysis of this aspect would be an interesting topic for future research. 
Even after 10 cycles over all time steps, the simple time-stepping scheme (M3) did not converge to the desired accuracy. Moreover, this method is sequential in time, and hence run-times largely exceed the other methods which can be executed parallel in time.

\bibliographystyle{splncs04}
\bibliography{references}

\begin{credits}
    \subsubsection{\ackname} 
    This work was supported by the DFG/FWF Collaborative Research Centre CREATOR (DFG: Project-ID 492661287/TRR 361; FWF: 10.55776/F90).    
    The authors would like to thank Hillary Fairbanks for providing access to Dane.
    
    \subsubsection{\discintname}
    The authors have no competing interests to declare that are
    relevant to the content of this article.
    \end{credits}

\end{document}